\documentclass{article}
\usepackage{amsmath,amscd,amsfonts,amssymb,amsthm}
\usepackage{graphicx}
\usepackage{epstopdf}
\usepackage{a4wide}
\usepackage{graphicx}
\usepackage{subfigure}
\usepackage[latin1]{inputenc}
\usepackage{color}
\usepackage{caption}
\usepackage{url}
\newcommand{\subj}[1]{\par\noindent{\bf Mathematics Subject Classification 2010: }#1.}
\newcommand{\keyw}[1]{\par\noindent{\bf Keywords: }#1.}
\theoremstyle{definition}

\newtheorem{theorem}{Theorem}

\newtheorem{example}{Example}

\theoremstyle{remark}

\def\a{\alpha}
\def\t{\tau}
\def\p{\psi}

\def\LD{{^CD_{a+}^{\a,\p}}}
\def\RLD{{D_{a+}^{\a,\p}}}
\def\LI{{I_{a+}^{\a,\p}}}
\def\LII{{I_{a+}^{2\a,\p}}}
\title{Fractional differential equations with mixed boundary conditions}

\author{Ricardo Almeida\\
\texttt{ricardo.almeida@ua.pt}}

\date{\small{Center for Research and Development in Mathematics and Applications (CIDMA)\\
Department of Mathematics, University of Aveiro, 3810--193 Aveiro, Portugal}}

\begin{document}
\maketitle
%--------------------------------------------------

\begin{abstract} In this paper, we discuss the  existence and uniqueness  of solutions of a boundary value problem for a fractional differential equation of order $\a\in(2,3)$, involving a general form of fractional derivative. First, we prove an equivalence between the Cauchy problem and the Volterra equation. Then, two results on the existence of solutions are proven, and we end with some illustrative examples.
\end{abstract}

\subj{26A33, 34A08, 47H10}

\keyw{Fractional differential equations, fractional calculus, fixed point theorems}

%--------------------------------------------------

\section{Introduction}

The aim of this paper is to investigate the existence and uniqueness  of solutions for a nonlinear Fractional Differential Equation (FDE) involving the Caputo fractional derivative of a function $x$ with respect to another function $\p$:
\begin{equation}\label{FDE}
\left\{
\begin{array}{l}
\LD x(t)=f(t,x(t)), \quad t\in[a,b]\\
x(a)=x_a\\
x'(a)=x^1_a\\
x(b)=K\LI x(v),
\end{array}
\right.
\end{equation}
where $\a\in(2,3)$, $x_a,x^1_a,K,v\in\mathbb R$ with $v\in(a,b]$, $x\in C^2[a,b]$, $\LD x$ is the $\psi$-Caputo fractional derivative of $x$ (see \eqref{def:caputo}) and $f:[a,b]\times\mathbb R\to\mathbb R$ is a continuous function. We assume that the restriction
\begin{equation}\label{def:N}N:=\frac{2K}{\Gamma(3+\a)}(\p(v)-\p(a))^{2+\a}-(\p(b)-\p(a))^2\not=0\end{equation}
is satisfied.

Fractional calculus can be seen as a generalization of ordinary calculus, in the sense that derivatives and integrals of arbitrary real or complex order are defined \cite{Kilbas,Malinowska,Samko}. Given a function $x$ with domain $[a,b]$, the Riemann--Liouville fractional integral of $x$ of order $\a>0$ is given by the formula
$${I_{a+}^{\a}} x(t):=\frac{1}{\Gamma(\a)}\int_a^t (t-\t)^{\a-1}x(\t)\,d\t.$$
When $\a$ is an integer number, the fractional integral is simply an $n-$tuple integral. With respect to differentiation, two of the most common definitions are the Riemann--Liouville fractional derivative
$${D_{a+}^{\a}} x(t):=  \left(\frac{d}{dt}\right)^n {I_{a+}^{n-\a}}x(t)$$
and the Caputo fractional derivative
$${^CD_{a+}^{\a}} x(t):={D_{a+}^{\a}} \left[x(t)-\sum_{k=0}^{n-1}\frac{x^{(k)}(a)}{k!}(t-a)^k\right],$$
where $n=[\a]+1$.
If $\a\in\mathbb N$, then ${D_{a+}^{\a}} x(t)={^CD_{a+}^{\a}} x(t)= x^{(n)}(t)$.

However, these are not the only definitions, and we can find many more in the literature. To overcome this problem, one solution is to present general definitions and then study some of their properties. This can be done e.g. considering fractional derivatives and fractional integrals with respect to another function $\p$ \cite[Section 2.5]{Kilbas}, and for particular choices of $\p$, we recover e.g. the Riemann--Liouville, the Hadamard, and the Erd\'{e}lyi--Kober fractional derivatives and fractional integrals. In \cite{Almeida}, a Caputo-type fractional derivative of a function with respect to another function is presented, and some of its main properties were studied.
FDE is a branch of mathematics, where in the differential equations, ordinary derivatives are replaced by fractional operators (see the seminal papers \cite{Agarwal,Diethelm,Lakshmikantham}). Since there are several definitions for fractional derivatives, we find in the literature several papers dealings with similar subjects for different operators: the Riemann-Liouville \cite{Ahmad,Didgar,Zhang}, the Caputo \cite{Agrawal2,YXu,Qarout}, the Hadamard \cite{Li, Wang2}, the Erdelyi-Kober \cite{Saqabi,Wang3}, etc. Fractional differential equations have been used to model more efficiently some real life phenomena like viscoelasticity \cite{Kempfle}, electrical circuits \cite{Lokenath}, electromagnetism \cite{Engheta}, in chemistry \cite{Douglas}, sound propagation \cite{Fellah}, etc, and has attracted the attention of numerous researchers.

In this paper, we continue the work started in \cite{Almeida}, by considering a FDE of order $\a\in(2,3)$, with mixed boundary conditions.
The paper is organized as follows. Section \ref{sec:FC} presents the notions that we will need, as well as some basic properties
of the operators, and Section \ref{sec:FPT} presents the two results on the existence of solution for the FDE \eqref{FDE}, using the Banach and the Leray--Schauder  fixed point theorems.

\section{Review on fractional calculus}\label{sec:FC}

We deal with fractional derivatives and fractional integrals with respect to another function. To fix terminology, a real $\a>0$ is the order of the fractional operators, $x:[a,b]\to\mathbb R$ is an integrable function, and $\p:[a,b]\to\mathbb R$ is an increasing differentiable function such that $\p'(t)\not=0$, for all $t\in [a,b]$.

The $\p$-Riemann--Liouville fractional integral of $x$ is given by
$$\LI x(t):=\frac{1}{\Gamma(\a)}\int_a^t \p'(\t)(\p(t)-\p(\t))^{\a-1}x(\t)\,d\t,$$
and the $\p$-Riemann--Liouville fractional derivative of $x$ by
\begin{align*}\RLD x(t)&:=  \left(\frac{1}{\p'(t)}\frac{d}{dt}\right)^n {I_{a+}^{n-\a,\p}}x(t)\\
&=\frac{1}{\Gamma(n-\a)}\left(\frac{1}{\p'(t)}\frac{d}{dt}\right)^n\int_a^t \p'(\t)(\p(t)-\p(\t))^{n-\a-1}x(\t)\,d\t,\end{align*}
where $n=[\a]+1$.
The semigroup law is valid for fractional integrals, that is,
$$\LI {I_{a+}^{\beta,\p}}x(t)={I_{a+}^{\a+\beta,\p}}x(t), \quad \forall \a,\beta>0.$$
The fractional derivative that we will deal in our work is a Caputo type operator, and it is defined by the expression
\begin{equation}\label{def:caputo}\LD x(t):=\RLD \left[x(t)-\sum_{k=0}^{n-1}\frac{x^{[k]}_\p(a)}{k!}(\p(t)-\p(a))^k\right],\end{equation}
where $x\in C^{n-1}[a,b]$, $n=[\a]+1$ for $\a\notin\mathbb N$, $n=\a$ for $\a\in\mathbb N,$ and
$$x^{[k]}_\p(t):=\left(\frac{1}{\p'(t)}\frac{d}{dt}\right)^kx(t).$$
If $x$ is of class $C^n$, then the $\p$-Caputo fractional derivative of $x$ is given by (cf. \cite[Theorem 3]{Almeida})
$$\LD x(t)= {I_{a+}^{n-\a,\p}}\left(\frac{1}{\p'(t)}\frac{d}{dt}\right)^nx(t).$$
Thus, if $\alpha=m\in\mathbb N$, we have
$$\LD x(t)=x^{[m]}_\p (t),$$
and for $\a\notin\mathbb N$, we have
$$\LD x(t)=\frac{1}{\Gamma(n-\a)}\int_a^t \p'(\t)(\p(t)-\p(\t))^{n-\a-1}x^{[n]}_\p (\t)\,d\t.$$
One important function in fractional calculus is the Mittag--Leffler function, which is  a generalization of the factorial function to real numbers, and it is given by $$E_\a(z):=\sum_{k=0}^\infty\frac{z^k}{\Gamma(k\a+1)}.$$
The derivative of a power function is given by
$$\LD (\p(t)-\p(a))^{\beta-1}=\frac{\Gamma(\beta)}{\Gamma(\beta-\a)}(\p(t)-\p(a))^{\beta-\a-1},$$
where  $\beta\in\mathbb R$ with $\beta>n$, and for $\lambda\in\mathbb R$, the derivative of the Mittag-Leffler function is
$$\LD E_\a(\lambda(\p(t)-\p(a))^\a)=\lambda E_\a(\lambda(\p(t)-\p(a))^\a).$$
The relation between fractional integrals and fractional derivatives is given next.

\begin{theorem}\label{relation}\cite{Almeida2} Let $x:[a,b]\to\mathbb R$ be a function.
\begin{enumerate}
\item\label{item1} If $x$ is continuous, then
$$\LD \LI x(t)=x(t).$$
\item\label{item2} If $x$ is of class $C^{n-1}$, then
$$\LI \LD x(t)=x(t)-\sum_{k=0}^{n-1}\frac{x^{[k]}_\p(a)}{k!}(\p(t)-\p(a))^k.$$
\end{enumerate}
\end{theorem}

We also refer to \cite{Almeida}, where Theorem 1 is formulated for another class of functions. For our work, we need the following result.

\begin{theorem} For $\beta>0$ and $\lambda\in\mathbb R\setminus \{0\}$, we have the two following formulas:
$$\LI (\p(t)-\p(a))^{\beta-1}=\frac{\Gamma(\beta)}{\Gamma(\beta+\a)}(\p(t)-\p(a))^{\beta+\a-1}$$
and
$$\LI E_\a(\lambda(\p(t)-\p(a))^\a)=\frac1\lambda \left(E_\a(\lambda(\p(t)-\p(a))^\a)-1\right).$$
\end{theorem}

\begin{proof} Direct computations lead to
\begin{align*} \LI (\p(t)-\p(a))^{\beta-1} & = \frac{1}{\Gamma(\a)}\int_a^t \p'(\t)(\p(t)-\p(\t))^{\a-1}(\p(\t)-\p(a))^{\beta-1}\,d\t\\
& = \frac{1}{\Gamma(\a)}(\p(t)-\p(a))^{\a-1}\int_a^t \p'(\t)\left(1-\frac{\p(\t)-\p(a)}{\p(t)-\p(a)}\right)^{\a-1}(\p(\t)-\p(a))^{\beta-1}\,d\t\\
& = \frac{1}{\Gamma(\a)}(\p(t)-\p(a))^{\beta+\a-1}\int_0^1(1-u)^{\a-1}u^{\beta-1}\,du\\
& = \frac{1}{\Gamma(\a)}(\p(t)-\p(a))^{\beta+\a-1}\frac{\Gamma(\beta)\Gamma(\a)}{\Gamma(\beta+\a)}
=\frac{\Gamma(\beta)}{\Gamma(\beta+\a)}(\p(t)-\p(a))^{\beta+\a-1}.
\end{align*}
For the second formula,
\begin{align*}
\LI E_\a(\lambda(\p(t)-\p(a))^\a)&= \sum_{k=0}^\infty\frac{\lambda^k}{\Gamma(k\a+1)}\LI (\p(t)-\p(a))^{k\a}\\
&= \sum_{k=0}^\infty\frac{\lambda^k}{\Gamma(k\a+\a+1)}(\p(t)-\p(a))^{k\a+\a}\\
& =\frac1\lambda \sum_{k=1}^\infty\frac{\lambda^k}{\Gamma(k\a+1)} (\p(t)-\p(a))^{k\a}
=\frac1\lambda \left(E_\a(\lambda(\p(t)-\p(a))^\a)-1\right).
\end{align*}
\end{proof}

%--------------------------------------------------

\section{Main Results}\label{sec:FPT}

Our next result establishes an equivalence between problem \eqref{FDE} and a Volterra integral equation.

\begin{theorem}\label{equiv}
Problem \eqref{FDE} is equivalent to the integral equation
\begin{equation}\label{FDEInt}x(t)=\LI f(t,x(t))+x_a+\frac{x^1_a}{\p'(a)}(\p(t)-\p(a))+\gamma_x(\p(t)-\p(a))^2, \quad t\in[a,b],\end{equation}
where
$$\gamma_x:=\Big[\LI f(b,x(b))-K \LII f(v,x(v))+x_a\Big(1-\frac{K}{\Gamma(1+\a)}(\p(v)-\p(a))^\a\Big)$$
\begin{equation}\label{gammax} +\frac{x^1_a}{\p'(a)}\Big(\p(b)-\p(a)-\frac{K}{\Gamma(2+\a)}(\p(v)-\p(a))^{1+\a}\Big)\Big] / N.\end{equation}
\end{theorem}

\begin{proof} Since $\LD x(t)=f(t,x(t))$, applying the fractional integral to both sides of the equation and using Theorem \ref{relation}, we arrive at the relation
$$x(t)=\LI f(t,x(t))+\gamma_0+\gamma_1(\p(t)-\p(a))+\gamma_2(\p(t)-\p(a))^2,$$
for some constants $\gamma_0,\gamma_1,\gamma_2$.
Since $x(a)=x_a$ and $x'(a)=x^1_a$, we deduce that $\gamma_0=x_a$ and $\gamma_1=x^1_a/\p'(a)$. To determine $\gamma_2$, first observe that
\begin{align*}
\LI x(t) &= \LI\Big( \LI f(t,x(t))+x_a+\frac{x^1_a}{\p'(a)}(\p(t)-\p(a))+\gamma_2(\p(t)-\p(a))^2 \Big)\\
&=\LII f(t,x(t))+\frac{x_a}{\Gamma(1+\a)}(\p(t)-\p(a))^\a+\frac{x^1_a}{\p'(a)\Gamma(2+\a)}(\p(t)-\p(a))^{1+\a}\\
& \quad +\frac{2\gamma_2}{\Gamma(3+\a)}(\p(t)-\p(a))^{2+\a}.\\
\end{align*}
Since $x(b)=K\LI x(v)$, we determine the value of $\gamma_2$, thus ending the first part of the proof. For the converse, if $x$ satisfies equation \eqref{FDEInt}, applying the fractional derivative to both sides of the equation, and attending that
$$\LD x_a=\LD(\p(t)-\p(a))=\LD(\p(t)-\p(a))^2=0,$$
we obtain that $\LD x(t)=f(t,x(t))$.
\end{proof}

Motivated by Theorem \ref{equiv}, we define the functional $F$ by
\begin{equation}\label{F}F(x)(t):=\LI f(t,x(t))+x_a+\frac{x^1_a}{\p'(a)}(\p(t)-\p(a))+\gamma_x(\p(t)-\p(a))^2,\end{equation}
where $\gamma_x$ is given by \eqref{gammax}. In the following, given a function $y:[c,d]\to\mathbb R$, we consider the usual norm
$\|y\|:=\sup_{t\in[c,d]}|y(t)|.$

\begin{theorem}\label{BanachT} If there exists  an integrable function $L:[a,b]\to\mathbb R^+_0$ such that
$$|f(t,x_1)-f(t,x_2)|\leq L(t)|x_1-x_2|, \quad \forall t\in[a,b], \, \forall x_1,x_2\in\mathbb R,$$
 there exists then a real $h>0$ (with $h\leq b-a$) for which the FDE \eqref{FDE} has a unique solution in $[a,a+h]$.
\end{theorem}

\begin{proof} Let $h>0$ be such that
\begin{equation}\label{contraction}\sup_{t\in[a,a+h]}\left\{ \LI L(t)+\Big[\LI L(b)+K\LII L(v)\Big]\frac{(\p(t)-\p(a))^2}{N} \right\}<1,\end{equation}
and $U$  the set
$$U:=\{x\in C^2[a,a+h]: \, \LD x(t) \, \mbox{ exists and is continous in } [a,a+h]\}.$$
It is enough to prove that $F:U\to U$  is a contraction.
Let us see that $F$ is well-defined, that is, $F(U)\subseteq U$. Given a function $x\in U$, we have that the map $t\mapsto F(x)(t)$ is of class $C^2$ and
$$\LD F(x)(t)=f(t,x(t))$$
is continuous. Now, let $x_1,x_2\in U$ be arbitrary. Then,
\begin{align*}\|F(x_1)-F(x_2)\|&\leq\sup_{t\in[a,a+h]}\Big\{\LI |f(t,x_1(t))-f(t,x_2(t))| +
\Big[\LI |f(b,x_1(b))-f(b,x_2(b))|\\
& \quad\quad  +K \LII |f(v,x_1(v))-f(v,x_2(v))|\Big]\frac{(\p(t)-\p(a))^2}{N}\Big\}\\
&\leq \sup_{t\in[a,a+h]}\Big\{ \LI L(t)+\Big[\LI L(b)+K\LII L(v)\Big]\frac{(\p(t)-\p(a))^2}{N}\Big\}\|x_1-x_2\|,
\end{align*}
which proves that $F$ is a contraction. Using the Banach fixed point theorem, problem \eqref{FDE} has a unique solution.
\end{proof}

Recall that a family of functions $\Upsilon$ is equicontinuous if for every $\epsilon>0$, there exists some $\delta>0$ such that $d(x(t_1),x(t_2))<\epsilon$ for all $x \in \Upsilon$ and all $t_1,t_2$ such that  $d(t_1,t_2)<\delta$.

\begin{theorem}\label{Leray-Schauder}  Suppose that there exist
\begin{enumerate}
\item two continuous functions  $F_1,F_2:[a,b]\to\mathbb R$ and a nondecreasing function $\Lambda:\mathbb R_0^+\to\mathbb R_0^+$ such that, for all $t\in[a,b]$ and $x\in\mathbb R$, we have
$$|f(t,x)|\leq F_1(t)\Lambda(|x|)+F_2(t),$$
\item a real $R>0$ such that
$$\omega_1\Lambda(R)+\omega_2+|x_a|+\frac{|x^1_a|}{\p'(a)}(\p(b)-\p(a))+\Big|x_a\Big[1-\frac{K}{\Gamma(1+\a)}(\p(v)-\p(a))^\a\Big]$$
$$+\frac{x^1_a}{\p'(a)}\Big[\p(b)-\p(a)-\frac{K}{\Gamma(2+\a)}(\p(v)-\p(a))^{1+\a}\Big]  \Big| \frac{(\p(b)-\p(a))^2}{N}\leq R,$$
where
$$w_i:= \sup_{t\in[a,b]} \LI |F_i(t)|+(\LI |F_i(b)| +K \LII |F_i(v)| )\frac{(\p(b)-\p(a))^2}{N}, \quad i=1,2.$$
\end{enumerate}
Then, the Cauchy problem \eqref{FDE} has at least one solution in $[a,b]$.
\end{theorem}

\begin{proof} Let $F$ be the functional given by the expression \eqref{F} and consider the closed ball
$$B_R:=\left\{ x\in C^2[a,b]\,:\, \| x\| \leq R  \right\}.$$
\textit{Step 1:} $F(B_R)\subseteq B_R$.

Given $x\in B_R$, we have the following
\begin{align*}\|F(x)\|& \leq \sup_{t\in[a,b]}\Big\{ \LI |f(t,x(t))|+|x_a|+\frac{|x^1_a|}{\p'(a)}(\p(t)-\p(a))+|\gamma_x|(\p(t)-\p(a))^2\Big\}\\
& \leq \sup_{t\in[a,b]} \LI |F_1(t)\Lambda(|x(t)|)+F_2(t)|+|x_a|+\frac{|x^1_a|}{\p'(a)}(\p(b)-\p(a)) +\Big[\LI |F_1(b)\Lambda(|x(b)|)+F_2(b)|\\
&\quad+K \LII |F_1(v)\Lambda(|x(v)|)+F_2(v)|+\Big|x_a\big[1-\frac{K}{\Gamma(1+\a)}(\p(v)-\p(a))^\a\big]\\
& \quad+\frac{x^1_a}{\p'(a)}\big[\p(b)-\p(a)-
\frac{K}{\Gamma(2+\a)}(\p(v)-\p(a))^{1+\a}\big]\Big|\Big]\frac{(\p(b)-\p(a))^2}{N}\\
& \leq\omega_1\Lambda(R)+\omega_2+|x_a|+\frac{|x^1_a|}{\p'(a)}(\p(b)-\p(a))+\Big|x_a\Big[1-\frac{K}{\Gamma(1+\a)}(\p(v)-\p(a))^\a\Big]\\
& \quad +\frac{x^1_a}{\p'(a)}\Big[\p(b)-\p(a)-\frac{K}{\Gamma(2+\a)}(\p(v)-\p(a))^{1+\a}\Big]  \Big| \frac{(\p(b)-\p(a))^2}{N}\leq R.
\end{align*}

\textit{Step 2:} $F(B_R)$ is equicontinuous.

Let $t_1,t_2\in[a,b]$ with $t_1>t_2$ and
$$M:=\max_{(t,x)\in[a,b]\times B_R}|f(t,x)|.$$
Then,
\begin{align*}|F(x)(t_1)-F(x)(t_2)|&\leq \frac{M}{\Gamma(\a+1)}\left[(\p(t_1)-\p(a))^\a-(\p(t_2)-\p(a))^\a \right]\\
& \quad +\frac{|x^1_a|}{\p'(a)}(\p(t_1)-\p(t_2))
+|\gamma_x|\left[(\p(t_1)-\p(a))^2-(\p(t_2)-\p(a))^2 \right]\end{align*}
which converges to zero as $t_1\to t_2$. By the Ascoli--Arzela Theorem, $F(B_R)$ is contained in a compact set.
Finally, we now prove that there are no $x\in \partial B_R$ and $\lambda\in(0,1)$ such that $x=\lambda F(x)$. For that purpose, suppose that such $x$ and $\lambda$ exist. In this case,
$$R=\|x\|< \|F(x)\|\leq R,$$
as we have seen before, which proves our claim. In conclusion, by the Leray--Schauder alternative, we prove the existence of a fixed point to $F$.
\end{proof}

% -------------------------------------------------------------------------------------------------------------------

\section{Examples}

\begin{example} Fix a kernel $\p:[0,1]\to\mathbb R$ and the fractional order $\a=2.5$. Consider the FDE
$$\left\{
\begin{array}{l}
{^CD_{0+}^{2.5,\p}} x(t)=\Gamma(4.5)(\p(t)-\p(0)), \quad t\in[0,1]\\
x(0)=0\\
x'(0)=0\\
x(1)=K{I_{0+}^{2.5,\p}} x(0.5),
\end{array}
\right.$$
where
$$K=\frac{6!(\p(1)-\p(0))^{3.5}}{\Gamma(4.5)(\p(0.5)-\p(0))^6}.$$
It is easy to check that the solution to this problem is the function $x(t)=(\p(t)-\p(0))^{3.5}$. We can apply Theorem \ref{BanachT}, by considering $L(t)=0$, for all $t$, and conclude about the existence and uniqueness of the solution in the interval $[0,1]$.
\end{example}

\begin{example} For our second example, consider the problem
$$\left\{
\begin{array}{l}
{^CD_{0+}^{2.5,\p}} x(t)=x(t), \quad t\in[0,1]\\
x(0)=1\\
x'(0)=0\\
x(1)=K{I_{0+}^{2.5,\p}} x(0.5),
\end{array}
\right.$$
where
$$K=\frac{E_{2.5}((\p(1)-\p(0))^{2.5})}{E_{2.5}((\p(0.5)-\p(0))^{2.5})-1}.$$
We assume that condition \eqref{def:N} is satisfied.
The solution is the function $x(t)=E_{2.5}((\p(t)-\p(0))^{2.5})$. We can apply Theorem \ref{BanachT}, by considering $L(t)=1$. Alternatively, applying Theorem \ref{Leray-Schauder} with $F_1(t)=1$, $F_2(t)=0$ and $\Lambda(x)=x$, we conclude about the existence of the solution on the interval [0,1].
\end{example}

\section{Conclusion}

In this paper we deal with fractional differential equations of order $\a\in(2,3)$, with initial and end-point conditions. First, we prove an equivalent problem by rewritten the FDE as a Volterra integral equation. Then, using some known fixed point theorems, we prove the existence and uniqueness of solution. Using the techniques exemplified in this paper, other fixed point theorems could be applied.
% ------------------------------------------------
\section*{Acknowledgments}

Work supported by Portuguese funds through the CIDMA - Center for Research and Development in Mathematics and Applications, and the Portuguese Foundation for Science and Technology (FCT-Funda\c{c}\~ao para a Ci\^encia e a Tecnologia), within project UID/MAT/04106/2013.  The author is very grateful to two anonymous referees, for valuable remarks and comments that improved this paper.

% ----------------------------------------

% ----------------------------------------

\end{document}